\documentclass[11pt,reqno]{amsart}
\usepackage[body={7in,9in},centering]{geometry}
\usepackage{amsmath,amssymb,amsthm,mathrsfs,extarrows}
\usepackage{color,xcolor}
\definecolor{cobalt}{RGB}{61,99,181}
\usepackage[colorlinks,citecolor=cobalt,linkcolor=cobalt]{hyperref}
\usepackage{float}
\usepackage{exscale}
\usepackage{relsize}
\usepackage{graphicx}
\usepackage{tikz,tikz-cd}

\date{\today}
\newtheorem{thm}{Theorem}[section]
\newtheorem{defi}[thm]{Definition}
\newtheorem{cor}[thm]{Corollary}
\newtheorem{lem}[thm]{Lemma}

\newtheorem{ex}[thm]{Example}

\numberwithin{equation}{section}
\makeatletter

\newcommand{\Rmnum}[1]{\expandafter\@slowromancap\romannumeral #1@}
\makeatother

\begin{document}

\title{The frequent hypercyclicity of unbounded operator}
\author[Xiongxun Huang]{Xiongxun Huang}

\author[Yonglu Shu]{Yonglu Shu}

\address[Xiongxun Huang]{College of Mathematics and Statistics, Chongqing University, Chongqing, 401331, P. R. China}
\address{Email:\href{mailto:2171819@cqu.edu.cn}{20171819@cqu.edu.cn}}
\address[Yonglu Shu]{College of Mathematics and Statistics, Chongqing University, Chongqing, 401331, P. R. China}
\address{Email: \href{mailto:shuyonglu@163.com}{shuyonglu@163.com}}

\keywords{unbounded; frequently hypercyclic; linear operator}


\begin{abstract}
We establish two Frequent Hypercyclicity Criteria for unbounded operators, inspired by the frameworks of Bayart–Grivaux and deLaubenfels–Emamirad–Grosse-Erdmann. These criteria simplify the verification and construction of frequently hypercyclic operators.
\end{abstract} \maketitle

\section{Introduction}\label{Intro}
Consider a topological vector space $X $ over the scalar field $\mathbb{K}$ (where $\mathbb{K}=\mathbb{R} $ or  $\mathbb{C} $). A continuous linear operator $T: X \rightarrow X $ is termed hypercyclic if there exists a vector $x \in X $ whose orbit under $T$, defined as 
$\operatorname{orb}(x, T)=\left\{T^n x \mid n=0,1,2, \ldots \ldots\right\} $, is dense in $X $. Such a vector $x $ is referred to as a hypercyclic vector for $T $, and the set of all hypercyclic vectors for $T $ is denoted by $H C(T) $ \cite{grosse2011linear}.

The study of dynamical properties of linear operators is a pivotal branch of functional analysis, focusing on the complexity of orbital structures under operator iterations, such as hypercyclicity, mixing, and chaos. This research not only deepens the theory of operators itself but also provides novel analytical tools for disciplines like differential equations and quantum mechanics. Since Birkhoff’s \cite{birkhoff1929demonstration} discovery in 1929 of the hypercyclicity of translation operators on entire function spaces, the dynamics of linear operators has evolved significantly: Maclane (1952)\cite{maclane1952sequences} demonstrated hypercyclicity for differentiation operators, while Rolewicz (1969)\cite{rolewicz1969orbits} revealed analogous properties for shift operators on Banach spaces. A milestone in this field was Kitai’s (1984)\cite{kitai1984invariant} hypercyclicity criterion in his doctoral thesis, which laid the foundation for systematic studies of mixing, chaos, and related properties.

Frequent hypercyclicity, as a strengthened form of hypercyclicity, was first introduced by Bayart et al. in 2006\cite{1}. It requires that the orbits of an operator intersect any nonempty open set with positive lower density, as following:

\begin{defi}[\cite{1}]
    Let  $X $ be a topological vector space and  $T: X \rightarrow X $ an operator. Then a vector  $x \in X $ is called frequently hypercyclic for  $T $ if, for every non-empty open subset  $U $ of  $X $, the set  $$\left\{n \in \mathbb{N}: T^n x \in U\right\} $$ has positive lower density. The operator  $T $ is called frequently hypercyclic if it possesses a frequently hypercyclic vector. The set of all frequently hypercyclic vectors for  $T $ is denoted by  $F H C(T) $.
\end{defi}

However, existing research predominantly focuses on bounded operators, leaving the dynamical properties of unbounded operators—ubiquitous in practical applications such as differentiation operators—largely unexplored. Although Bes et al. (2001)\cite{bes2001chaotic} pioneered the study of hypercyclicity for unbounded operators and verified the chaoticity of differentiation operators, a criterion for frequent hypercyclicity in the unbounded case remains absent.

This paper aims to bridge this theoretical gap. By rigorously analyzing the proof structure of the bounded operator criterion, we identify that the properties of closed operators combined with appropriate restrictions on dense subsets provide a critical breakthrough for the unbounded setting. Building on this insight, we propose a new criterion for frequent hypercyclicity of unbounded linear operators, extending Bayart’s classical results to a broader operator class. This work not only enriches the theoretical framework of linear dynamical systems but also offers new analytical tools for investigating the complex dynamical behaviors of unbounded operators, such as differentiation operators.

We start by extending the concept of frequent hypercyclicity to unbounded situations (§2). Then we will give and prove the frequent hypercyclicity criterion for closed operators. We will also use the criterion to find several types of such operators. In §3 we will give another unbounded Frequent Hypercyclicity Criterion. Then the criterion implies the Frequent Hypercyclicity Criterion for unbounded semigroups  $\{e^{tA}\}_{t\geq0} $.

\section{Closed densely defined operator}\label{S2}
A pivotal criterion guaranteeing the frequent hypercyclicity of linear operators was established by Bayart and Grivaux \cite{1}. Then A. Bonilla and K.-G. Grosse-Erdmann derive a strengthened form\cite{bonilla2007frequently}. And we derive here a undbounded form of the criterion of A. Bonilla and K.-G. Grosse-Erdmann.

We recall a $F$-space refers to a topological vector space endowed with a complete metric that is translation-invariant. Its topology can equivalently be defined by an $F$-norm, which is a functional $\|\cdot\|: X \to \mathbb{R}_+$ satisfying the following axioms for all $x, y \in X$ and $\lambda \in \mathbb{K}$:
\begin{enumerate}
    \item[(i)] Positive definiteness: $\|x\| = 0 \iff x = 0$;
    \item[(ii)] Subhomogeneity: $\|\lambda x\| \leq \|x\|$ whenever $|\lambda| \leq 1$;
    \item[(iii)] Continuity in scalar: $\lim_{\lambda \to 0} \|\lambda x\| = 0$;
    \item[(iv)] Subadditivity: $\|x + y\| \leq \|x\| + \|y\|$.
\end{enumerate}
See \cite{kalton1984f} for equivalent formulations.

Let $X$  be an $F$ -space equipped with an $F$ -norm $\|\cdot\|$. A series $\sum_{k=1}^{\infty} x_k  $, is called unconditionally convergent if for every  $\varepsilon>0 $ there is some  $N \geq 1 $ such that for every finite subset  $F \subset \mathbb{N} $ with  $F \cap\{1,2, \ldots, N\}=\emptyset $ we have
 $$
\left\|\sum_{k \in F} x_k \right\|<\varepsilon .
 $$

Let  $T $ be an unbounded operator on a separable  $F $-space  $X. $ Since  $T $ is only densely defined on  $X $, it's possible for a vector  $f $ to be within the domain of  $T $, yet the result of  $T $ applied to  $f $, denoted as  $Tf $, might not be in the domain of  $T $. For this reason, we define  $T $ to be frequently hypercyclic as following: 
\begin{defi}
 $T $ is frequently hypercyclic if there is a vector  $f $  in the domain of  $T $ such that for every integer  $n\geq1 $ the vector  $T^nf $ is in the domain of  $T $ and every non-empty open subset  $U $ of  $X $, the set  $$\left\{n \in \mathbb{N}: T^n f \in U\right\} $$ has positive lower density. In this case, the vector  $f $ is called a frequently hypercyclic vector of  $T. $ 
\end{defi}
Frequently hypercyclic vectors can be constructed for an unbounded operator  $T $, under a sufficient condition analogous to the frequently hypercyclic criterion. The construction in the proof of the subsequent theorem incorporates a method provided by A. Bonilla and K.-G. Grosse-Erdmann\cite{bonilla2007frequently}. 
\begin{thm}[Frequent Hypercyclicity Criterion]\label{thm:result}\index{Theorem on some result}
Let  $A $ be a closed operator on a separable  $F $-space  $X $ for which  $A^r $ is a closed operator for all positive integer  $r $. Suppose that there is a dense subset  $X_0 $ of the domain of  $T $ and a mapping  $B: X_0 \rightarrow X_0 $ such that:
\begin{enumerate}
    \item  $ABx=x $ for all  $x \in X_0 $;
    \item  $\sum_{n=1}^{\infty} A^n x $ converges unconditionally, for all  $x \in X_0 $;
    \item  $\sum_{n=1}^{\infty} B^n x $ converges unconditionally, for all  $x \in X_0 $.
\end{enumerate}
Then  $A $ is frequently hypercyclic.
\end{thm}

We shall give the proof of Theorem \ref{thm:result} in a more general context. The concept of hypercyclicity can be extended to a notion of universality by substituting the sequence of operator iterations  $(T^n) $ with arbitrary sequence of mappings  $(T_n) $. Here, it's permissible to have different domain and range spaces.
\begin{defi}
Let $ X,Y $ be topological vector spaces and $ (T_n)_{n\in\mathbb{N}} $ a sequence of mappings from $X$ to $Y$. Then the sequence  $\left(T_n\right) $ is called frequently universal if, for every non-empty open subset  $U $ of  $Y $,there exists a vector $x \in X $,  $$\underline{\text{dens}}\left\{n \in \mathbb{N}: T_n x \in U\right\}>0. $$ Such a vector $x $ is referred to as a frequently universal element for $T $. The set of all frequently universal elements for  $\left(T_n\right) $ is denoted by  $F U\left(T_n\right) $.
\end{defi}

We also need the following concept. 
\begin{defi}[\cite{bonilla2007frequently}]
Let $X$  be an $F$ -space equipped with an $F$ -norm $\|\cdot\|$, a family of series  $\sum_{k=1}^{\infty} x_{k, j}, j \in I $, is said to be unconditionally convergent, uniformly in  $j \in I $, if for every  $\varepsilon>0 $ there is some  $N \geq 1 $ such that for every finite set  $F \subset \mathbb{N} $ with  $F \cap\{1,2, \ldots, N\}=\emptyset $ and every  $j \in I $ we have
 $$
\left\|\sum_{k \in F} x_{k, j}\right\|<\varepsilon .
 $$
\end{defi}
In the sequel we will need the following lemma.
\begin{lem}[\cite{bonilla2007frequently}]\label{lemma 2.4}
Let  $\rho ( l, \nu ) , l, \nu \geq 1 $, be non-negative numbers such that
 $$\sum_{l,v=1}^{\infty}\frac{\rho(l,v)}{2^{l+v}}<\infty. $$
Then there are pairwise disjoint sets  $A( l, \nu ) \subset \mathbb{N} , l, \nu \geq 1 $, of positive lower density such that, for all  $n\in A( l, \nu ) , m\in A( k, \mu ) $, $$\begin{aligned}n&\geq\rho(l,\nu),\\|n-m|&\geq\rho(l,\nu)+\rho(k,\mu)\quad \text { if } n\neq m.\end{aligned} $$
\end{lem}
Theorem \ref{thm:result} is a special case of the following theorem.And the proof of the latter will be an adaptation of that of A. Bonilla and K.-G. Grosse-Erdmann\cite{bonilla2007frequently}.
\begin{thm}[Frequent Hypercyclicity Criterion]\label{thm:result1}\index{Theorem on some result}
Let  $X $ be an  $F $-space,  $Y $ a separable  $F $-space and  $T_n: X \rightarrow Y, n \in \mathbb{N} $, closed operators. Define  $D(T_{\infty}):=\bigcap_{n=0}^{\infty}D(T_{n}) $. Suppose that there exists a topologically dense subset  $Y_0 $ of  $Y $ and a sequence of mappings  $S_n: Y_0 \rightarrow D(T_{\infty}), n \in \mathbb{N} $, such that:
\begin{enumerate}
    \item  $\sum_{n=1}^{\infty} T_{k+n} S_k y $ converges unconditionally in  $Y $, uniformly in  $k \in \mathbb{N} $, for all  $y \in Y_0 $;
    \item  $\sum_{n=1}^{\infty} T_k S_{k+n} y $ converges unconditionally in  $Y $, uniformly in  $k \in \mathbb{N} $, for all  $y \in Y_0 $;
    \item  $\sum_{n=1}^{\infty} S_n y $ converges unconditionally in  $X $, for all  $y \in Y_0 $;
    \item  $T_n S_n y \rightarrow y $ for all  $y \in Y_0 $.
\end{enumerate}
Then the sequence  $\left(T_n\right) $ is frequently universal.
\end{thm}
\begin{proof}
Since  $Y $ is separable we may assume that  $Y_0 $ is countable,
 $$
Y_0=\left\{y_1, y_2, y_3, \ldots\right\}
 $$
The conditions (1)-(4) imply that there are  $N_l \in \mathbb{N}, l \geq 1 $, such that, for all  $\lambda \leq l $, all  $k \in \mathbb{N} $ and all finite subsets  $F $ of  $\mathbb{N} $ with  $F \cap\left\{1,2, \ldots, N_l-1\right\}=\emptyset $ we have

\begin{equation}\label{1}
\left\|\sum_{n \in F} T_{k+n} S_k y_\lambda\right\| \leq \frac{1}{l 2^l},
\end{equation}
\begin{equation}\label{2}
\left\|\sum_{n \in F} T_k S_{k+n} y_\lambda\right\| \leq \frac{1}{l 2^l}, 
\end{equation}
\begin{equation}\label{3}
\left\|\sum_{n \in F} S_n y_l\right\| \leq \frac{1}{2^l}, 
\end{equation}
\begin{equation}\label{4}
\left\|T_n S_n y_l-y_l\right\| \leq \frac{1}{2^l} \quad \text { for } n \geq N_l.
\end{equation}
We let  $\rho(l, \nu)=\nu $ for  $l, \nu \geq 1 $; then by Lemma \ref{lemma 2.4} there are pairwise disjoint sets  $A(l, \nu), l, \nu \geq 1 $, of positive lower density with the properties stated there. We define
\begin{equation}\label{5}
z_n= \begin{cases}y_l & \text { if } n \in A\left(l, N_l\right) \\ 0 & \text { otherwise }\end{cases}
\end{equation}
and set
\begin{equation}\label{6}
x=\sum_{n=1}^{\infty} S_n z_n
\end{equation}
To prove that \eqref{6} converges in $ X $, observe that for all $ l \geq 1 $,
\begin{equation}\label{7}
\sum_{n \in A\left(l, N_l\right)} S_n z_n = \sum_{n \in A\left(l, N_l\right)} S_n y_l
\end{equation}
By condition 3, \eqref{7} converges unconditionally. 
According to Lemma \ref{lemma 2.4}, we have 
$$
n \geq N_l \quad \text{for any } n \in A\left(l, N_l\right),
$$
and then from \eqref{3}, for any finite subset $ F \subset A(l, N_l) $,
$$
\left\|\sum_{n \in F} S_n y_l\right\| \leq \frac{1}{2^l}.
$$
Consequently, the same inequality holds for the entire $ A(l, N_l) $:
$$
\left\|\sum_{n \in A(l, N_l)} S_n y_l\right\| \leq \frac{1}{2^l}.
$$
For any $ \varepsilon > 0 $, there exists $ L \in \mathbb{N} $ such that 
$$
\sum_{l=L+1}^{\infty} \frac{1}{2^l} < \frac{\varepsilon}{2}.
$$
Since each series $ \sum_{n \in A\left(l, N_l\right)} S_n y_l $ converges unconditionally, for each $ l \leq L $, there exists $ N_l \in \mathbb{N} $ such that 
$$
\left\|\sum_{n \in A\left(l, N_l\right) \setminus \left\{1, \dots, N_l\right\}} S_n y_l\right\| < \frac{\varepsilon}{2^{l+1}}.
$$
Let $ N = \max \left\{N_l \mid 1 \leq l \leq L\right\} $. For $ m > N $,
$$
\begin{aligned}
\left\|\sum_{n=N+1}^m S_n z_n\right\| &\leq \left\| \sum_{l=1}^L \sum_{n \in A\left(l, N_l\right) \setminus \left\{1, \dots, N\right\}} S_n y_l \right\| + \sum_{l=L+1}^{\infty} \left\| \sum_{n \in A\left(l, N_l\right)} S_n y_l \right\| \\
&\leq \sum_{l=1}^L \frac{\varepsilon}{2^{l+1}} + \sum_{l=L+1}^{\infty} \frac{1}{2^l} \\
&< \frac{\varepsilon}{2} + \frac{\varepsilon}{2} = \varepsilon.
\end{aligned}
$$
This establishes the unconditional convergence of \eqref{6}.

Now fix  $n \in A\left(l, N_l\right) $ for some  $l \in \mathbb{N} $. Then we consider the three summands in turn.
\begin{equation}\label{8}
\sum_{j<n} T_n S_j z_j+\sum_{j>n} T_n S_j z_j+\left(T_n S_n z_n-y_l\right)
\end{equation}
We have with \eqref{5}
\begin{equation}\label{9}
\sum_{j<n} T_n S_j z_j=\sum_{\lambda=1}^l \sum_{\substack{j \in A\left(\lambda, N_\lambda\right) \\ j<n}}  T_n  S_j y_\lambda +\sum_{\lambda=l+1}^{\infty} \sum_{\substack{j \in A\left(\lambda, N_\lambda\right) \\ j<n}} T_n S_j y_\lambda .
\end{equation}
Since, by Lemma \ref{lemma 2.4},
 $$
|n-j| \geq N_l,
 $$
we have by \eqref{1} for  $\lambda \leq l $
\begin{equation}\label{10}
\left\|\sum_{\substack{j \in A\left(\lambda, N_\lambda\right) \\ j<n}} T_n S_j y_\lambda\right\|=\left\|\sum_{\substack{j \in A\left(\lambda, N_\lambda\right) \\ j<n}} T_{j+(n-j)} S_j y_\lambda\right\| \leq \frac{1}{l2^l} ;
\end{equation}
since also
 $$
|n-j| \geq N_\lambda \quad \text { for } j \in A\left(\lambda, N_\lambda\right),
 $$
we have for all  $\lambda $
\begin{equation}\label{11}
\left\|\sum_{\substack{j \in A\left(\lambda, N_\lambda\right) \\ j<n}} T_n S_j y_\lambda\right\|=\left\|\sum_{\substack{j \in A\left(\lambda, N_\lambda\right) \\ j<n}} T_{j+(n-j)} S_j y_\lambda\right\| \leq \frac{1}{\lambda 2^\lambda} .
\end{equation}
Hence, \eqref{9}-\eqref{11} imply that
\begin{equation}\label{12}
\left\|\sum_{j<n} T_n S_j z_j\right\| \leq \sum_{\lambda=1}^l \frac{1}{l 2^l}+\sum_{\lambda=l+1}^{\infty} \frac{1}{\lambda 2^\lambda} \leq \frac{2}{2^l} .
\end{equation}
In essentially the same way, using \eqref{2} instead of \eqref{1}, we see that, for all  $M>n $,
 $$
\left\|\sum_{n<j \leq M} T_n S_j z_j\right\| \leq \frac{2}{2^l},
 $$
hence
\begin{equation}\label{13}
\left\|\sum_{j>n} T_n S_j z_j\right\| \leq \frac{2}{2^l} .
\end{equation}
Finally, by \eqref{4} and \eqref{5} we have
\begin{equation}\label{14}
\left\|T_n S_n z_n-y_l\right\|=\left\|T_n S_n y_l-y_l\right\| \leq \frac{1}{2^l}
\end{equation}
since  $n \geq N_l $.
Altogether, \eqref{12}-\eqref{14} tell us that
 $$
\left\|\sum_{j<n} T_n S_j z_j+\sum_{j>n} T_n S_j z_j+\left(T_n S_n z_n-y_l\right) \right\|\leq \frac{5}{2^l} <\infty\quad \text { if } n \in A\left(l, N_l\right)
 $$
Since  $T_n $ is closed,  $$T_nx-y_l=\sum_{j<n} T_n S_j z_j+\sum_{j>n} T_n S_j z_j+\left(T_n S_n z_n-y_l\right). $$
And hence  $$\left\|T_n-y_l\right\|\leq \frac{5}{2^l}\quad \text { if } n \in A\left(l, N_l\right).  $$
Since $ Y_0 $ is dense and $ 5 / 2^l \rightarrow 0 $ as $ l \rightarrow \infty $, for any open set $ U $ in $ X $, there exists some $ l \in \mathbb{N} $ such that  
$$
A\left(l, N_l\right) \subset \left\{ n \in \mathbb{N} : T_n x \in U \right\}.
$$
Since $ A\left(l, N_l\right) $ has positive lower density, we have  
$$
\underline{\operatorname{dens}}\left(\left\{ n \in \mathbb{N} : T_n x \in U \right\}\right) > 0.
$$
This shows that  $\left(T_n\right) $ is frequently universal.
\end{proof}

\begin{cor}  
Let  $A $ be a frequently hypercyclic operator on a (real or complex) infinite-dimensional separable Banach space  $ X  $ satisfying Theorem \ref{thm:result}. If  $ A^{-1}  $ is a bounded linear operator on  $ X  $, then  $ A^{-1}  $ is also frequently hypercyclic.
\end{cor}

\begin{proof}
Suppose $ A  $ is a frequently hypercyclic operator on a (real or complex) infinite-dimensional separable Banach space  $ X  $ satisfying Theorem \ref{thm:result}, and  $ A^{-1}  $ is a bounded linear operator on  $ X  $. By condition 1 of Theorem \ref{thm:result}, we have
 $$
\forall x \in X_0: A^{-1}x = A^{-1}(ABx) = (A^{-1}A)Bx = Bx,
 $$
which shows that  $ B  $ coincides with the restriction of  $ A^{-1}  $ to  $ X_0  $.

According to conditions 1, 2, and 3 of Theorem \ref{thm:result}, we obtain:
\begin{enumerate}
    \item For every  $ x \in X_0  $:
     $$
    A^{-1} A x = x
     $$
    
    \item For every  $ x \in X_0  $:
     $$
    \sum_{n=1}^{\infty} (A^{-1})^n x = \sum_{n=1}^{\infty} B^n x
     $$
    is unconditionally convergent in  $ X  $.
    
    \item For every  $ x \in X_0  $:
     $$
    \sum_{n=1}^{\infty} A^n x
     $$
    is unconditionally convergent in  $ X  $.
\end{enumerate}
Therefore, by Theorem \ref{thm:result},  $ A^{-1}  $ is frequently hypercyclic.
\end{proof}

\begin{cor}
Let  $ A  $ be a frequently hypercyclic operator on a (real or complex) infinite-dimensional separable Banach space  $ X  $ satisfying Theorem \ref{thm:result}. Then all its powers  $ A^n  $ ( $ n \in \mathbb{N}  $) are frequently hypercyclic operators.
\end{cor}

\begin{proof}
Suppose  $ A  $ is a frequently hypercyclic operator on a (real or complex) infinite-dimensional separable Banach space  $ X  $ satisfying Theorem \ref{thm:result}. Then for any  $ n \in \mathbb{N}  $, its power  $ A^n  $ also satisfies these conditions.

Specifically, fix  $ n \in \mathbb{N}  $. For the dense subset  $ X_0  $, we have
 $$
X_0 \subseteq D^{\infty}(A) = D^{\infty}\left(A^n\right),
 $$
and the map  $ B^n: X_0 \rightarrow X_0  $ satisfies the required conditions. By conditions 1, 2, and 3 of Theorem \ref{thm:result}, we obtain:
\begin{enumerate}
    \item For every  $ x \in X_0  $:
     $$
    A^n B^n x = A^{n-1}(A B^n x) = A^{n-1}((A B) B^{n-1} x) = A^{n-1} B^{n-1} x = \cdots = A B x = x
     $$
    
    \item For every  $ x \in X_0  $:
     $$
    \sum_{m=1}^{\infty} (A^n)^m x = \sum_{m=1}^{\infty} A^{m n} x
     $$
    is unconditionally convergent in  $ X  $.
    
    \item For every  $ x \in X_0  $:
     $$
    \sum_{m=1}^{\infty} (B^n)^m x = \sum_{m=1}^{\infty} B^{m n} x
     $$
    is unconditionally convergent in  $ X  $.
\end{enumerate}
Therefore, by Theorem \ref{thm:result}, all power  $ A^n  $ ( $ n \in \mathbb{N}  $) is frequently hypercyclic.
\end{proof}

\begin{cor}

Let  $ A  $ be a frequently hypercyclic operator on a (real or complex) infinite-dimensional separable Banach space  $ X  $ satisfying Theorem \ref{thm:result}. Then for any (real or complex) scalar  $ \lambda  $ with  $ |\lambda|=1  $, the operator  $ \lambda A  $ is frequently hypercyclic.
\end{cor}

\begin{proof}
Suppose  $ A  $ is a frequently hypercyclic operator on a (real or complex) infinite-dimensional separable Banach space  $ X  $ satisfying Theorem \ref{thm:result}. Then all operators  $ \lambda A  $ ( $ |\lambda|=1  $) are frequently hypercyclic.

Specifically, consider the dense subset  $ X_0  $. We have
 $$
X_0 \subseteq D^{\infty}(A) = D^{\infty}(\lambda A),
 $$
and the map  $ \lambda^{-1}B: X_0 \rightarrow X_0  $ satisfies the required conditions. By conditions 1, 2, and 3 of Theorem \ref{thm:result}, we obtain:
\begin{enumerate}
    \item For every  $ x \in X_0  $:
     $$
    (\lambda A)(\lambda^{-1}B)x = (\lambda \lambda^{-1})A B x = A B x = x
     $$
    
    \item For every  $ x \in X_0  $:
     $$
    \sum_{n=1}^{\infty} (\lambda A)^n x = \sum_{n=1}^{\infty} \lambda^n A^n x
     $$
    is unconditionally convergent in  $ X  $.
    
    \item For every  $ x \in X_0  $:
     $$
    \sum_{n=1}^{\infty} (\lambda^{-1} B)^n x = \sum_{n=1}^{\infty} \lambda^{-n} B^n x
     $$
    is unconditionally convergent in  $ X  $.
\end{enumerate}
Therefore, by Theorem \ref{thm:result}, each operator  $ \lambda A  $ ( $ |\lambda|=1  $) is frequently hypercyclic.
\end{proof}

\begin{ex}[]
The operator  $D $ of differentiation defined on the dense, linear submanifold  $\mathcal{E} \equiv\left\{f \in H^2: f^{\prime} \in H^2\right\} $ of  $H^2 $ by  $D: f \mapsto f^{\prime} $ is frequently hypercyclic.
\end{ex}
\begin{proof}
Consider the set of polynomials as $X_0$ in Theorem \ref{thm:result}, which is dense in $ H^2(\mathbb{D}) $. Define the operator $ B $ on $ X_0 $ by  
$$(B f)(z) = \int_0^z f(t)  dt.$$  
With this definition, conditions 1 and 2 of Theorem \ref{thm:result} are satisfied. Since polynomials are finite sums of monomials, it suffices to verify condition 3 for all monomials $ f_k(z) = z^k $, $ k \geq 0 $. For these monomials, we have  
$$
\sum_{n=1}^{\infty} B^n f_k(z) = \sum_{n=1}^{\infty} \frac{k!}{(k+n)!} z^{k+n},
$$  
which is clearly absolutely convergent and therefore unconditionally convergent in $ H^2(\mathbb{D}) $.
\end{proof}

\begin{ex}
For any  $ k \in \mathbb{N} $, the differentiation operator  $ D: f \mapsto f' $ on  $ C^k([a,b]) $ and its iterations  $ D^r $ are unbounded closed operators and frequently hypercyclic.
\end{ex}

\begin{proof}
We prove by mathematical induction that  $ D^r $ is closed for all  $ r \geq 1 $.

Base case ( $ r = 1 $): Let  $ \{f_n\} \subset C^{1+k}([a, b]) $ satisfy:
\begin{itemize}
    \item $ f_n \to f $ in $ C^k([a, b]) $
    \item $ Df_n \to g $ in $ C^k([a, b]) $
\end{itemize}
By interchangeability of differentiation and uniform convergence, we have  $ f \in C^{1+k}([a, b]) $ and $ Df = g $. Thus  $ D $ is closed.

Inductive step: Assume $ D^m $ is closed for some $ m \geq 1 $. We prove $ D^{m+1} $ is closed. Let $ \{f_n\} \subset C^{m+k+1}([a, b]) $ satisfy:
\begin{itemize}
    \item $ f_n \to f $ in $ C^k([a, b]) $
    \item $ D^{m+1}f_n \to g $ in $ C^k([a, b]) $
\end{itemize}

Applying the inductive hypothesis to $ D^m $:
\begin{itemize}
    \item Since $ f_n \to f $ in $ C^k([a, b]) $ and $ D^m f_n \to h $ in $ C^k([a, b]) $
    \item By closure of $ D^m $, we have $ f \in C^{m+k}([a, b]) $ and $ D^m f = h $
\end{itemize}

Observe that:
\begin{itemize}
    \item $ D^{m+1}f_n = D(D^m f_n) \to g $ in $ C^k([a, b]) $
    \item $ D^m f_n \to h $ in $ C^k([a, b]) $
\end{itemize}
By closure of $ D $, we obtain $ h \in C^{1+k}([a, b]) $ and $ Dh = g $. Hence $ f \in C^{m+k+1}([a, b]) $ and $ D^{m+1}f = g $.

 Take $ X_0 $ as the polynomial space which is dense in $ C^k([a,b]) $. Define the integration operator $ B $ on $ X_0 $ by:
 $$
(Bf)(x) = \int_a^x f(t) dt
 $$
Conditions (1) and (2) of Theorem \ref{thm:result} are satisfied. For condition (3), consider monomials $ f_k(x) = x^k $:
 $$
\sum_{n=1}^\infty B^n f_k(x) = \sum_{n=1}^\infty \frac{k!}{(k+n)!} x^{k+n}
 $$
This series converges absolutely (hence unconditionally) in $ C^k([a,b]) $.
\end{proof}
We now consider weighted shift operators. 
\begin{ex}[]
Let $ X$ be the (real or complex) sequence space $ \ell_p $ ($ 1 \leq p < \infty $) or $ c_0 $, the weighted backward shift operator $ A $ is defined by  
$$
x :=(x_k)_{k \in \mathbb{N}} \mapsto Ax := \left( w^k x_{k+1} \right)_{k \in \mathbb{N}},
$$  
where $ w \in \mathbb{R} $ or $ w \in \mathbb{C} $ with $ |w| > 1 $, and its domain is given by  
$$
D(A) := \left\{ (x_k)_{k \in \mathbb{N}} \in X \, \bigg| \, \left( w^k x_{k+1} \right)_{k \in \mathbb{N}} \in X \right\}.
$$  
Then $ A$ is an unbounded frequently hypercyclic operator.

\end{ex}
\begin{proof}
Let $\mathbb{F} := \mathbb{R}$ or $\mathbb{F} := \mathbb{C}$, and let $\|\cdot\|$ denote the norm on $X$. For each $n \in \mathbb{N}$, the linear operator  
$$
A^n(x_k)_{k \in \mathbb{N}} = \left( w^{(k+n-1) + (k+n-2) + \cdots + k} x_{k+n} \right)_{k \in \mathbb{N}}, (x_k)_{k \in \mathbb{N}} \in D(A^n),
$$  
is densely defined, with its domain given by  
$$
D(A^n) := \left\{ (x_k)_{k \in \mathbb{N}} \in X \, \bigg| \, \left( w^{(k+n-1) + (k+n-2) + \cdots + k} x_{k+n} \right)_{k \in \mathbb{N}} \in X \right\}.
$$

\begin{enumerate}
    \item Unbounded:
    Construct a sequence $x^{(n)} \in D(A)$, where $x_{k+1}^{(n)} = \delta_{k,n} / w^k$ (the $(n+1)$-th term is $1 / w^n$, and the rest are 0). Then $\left\|x^{(n)}\right\|_X = |w|^{-n}$ (in $\ell^p$) or tends to 0 (in $c_0$). However, the $n$-th term of $A x^{(n)}$ is $w^n \cdot (1 / w^n) = 1$, and the rest are 0, so $\left\|A x^{(n)}\right\|_X = 1$. Normalizing $x^{(n)}$ to a unit vector (by multiplying by $|w|^n$), we get $\left\|A x^{(n)}\right\|_X = |w|^n \rightarrow \infty$ (since $|w| > 1$), proving that $A$ is unbounded.

    \item Closed:  
    Suppose $x^{(n)} \rightarrow x$ in $X$ and $A^r x^{(n)} \rightarrow y$ in $X$. For each component $k$, we have:  
    $$
    \left(A^r x^{(n)}\right)_k = w^{(k+r-1) + (k+r-2) + \cdots + k} x_{k+r}^{(n)} \rightarrow y_k  (n \rightarrow \infty).
    $$  
    Since $x^{(n)} \rightarrow x$, it follows that $x_{k+r}^{(n)} \rightarrow x_{k+r}$, and hence $y_k = w^{(k+r-1) + \cdots + k} x_{k+r}$. Since $y \in X$, we have $\left(w^{(k+r-1) + \cdots + k} x_{k+r}\right)_{k \in \mathbb{N}} \in X$, i.e., $x \in D(A^r)$ and $A^r x = y$. Thus, the graph of $A^r$ is closed.
\end{enumerate}

The domain $D(A^n)$ contains the dense subspace $c_{00}$ in $(X, \|\cdot\|)$. Defining  
$$
D^\infty(A) := \bigcap_{n=0}^\infty D(A^n) \supset c_{00},
$$  
it is clear that $D^\infty(A)$ is a dense subspace in $(X, \|\cdot\|)$.

Define $B$ as a bounded right inverse of $A$:  
$$
B(x_k)_{k \in \mathbb{N}} := \left(0, w^{-1}x_1, w^{-2}x_2, \ldots\right),  (x_k)_{k \in \mathbb{N}} \in X.
$$  
For any $n \in \mathbb{N}$,  
$$
B^n(x_k)_{k \in \mathbb{N}} = (\underbrace{0, \ldots, 0}_{n \text{ terms}}, w^{-[n + (n-1) + \cdots + 1]}x_1, w^{-[(n+1) + n + \cdots + 2]}x_2, \ldots).
$$  

By assumption, the set $X_0 = \mathrm{span}\{e_n : n \in \mathbb{N}\}$ is dense in $X$. For any $x \in X_0$, the series $\sum_{n=1}^\infty A^n x$ is finite and therefore unconditionally convergent. Clearly, $ABx = x$ for all $x \in X$.  

Furthermore, for $k \in \mathbb{N}$,  
$$
\begin{aligned}
\sum_{n=1}^\infty B^n e_k &= \sum_{n=1}^\infty \frac{1}{w^{k + (k+1) + \cdots + (k+n-1)}} e_{k+n} \\
&= w^{1 + 2 + \cdots + (k-1)} \sum_{n=1}^\infty \frac{1}{w^{1 + 2 + \cdots + (k+n-1)}} e_{k+n}.
\end{aligned}
$$  

Since $|w| > 1$, this series converges absolutely in $X$ and is thus unconditionally convergent. Similarly, $\sum_{n=1}^\infty B^n x$ is unconditionally convergent for any $x \in X_0$. Finally, applying Theorem \ref{thm:result}, we conclude that $A$ is an unbounded frequently hypercyclic operator.

\end{proof}

\section{Unbounded operator and semi-group}\label{TA}

The Frequent Hypercyclicity Criterion, which serves as a sufficient condition for frequent hypercyclicity, was discovered independently by Bayart and Grivaux \cite{1}. A broader version of this finding was later provided by A. Bonilla and K.-G. Grosse-Erdmann \cite{bonilla2007frequently}. For our subsequent discussion, we will use the following version of the criterion tailored for sequences of operators.
\begin{lem}[\cite{bonilla2007frequently}]\label{lemma 3.1}
Let  $X ,Y $ be separable Banach space and  $T_n: X \rightarrow Y, n \in \mathbb{N} $, bounded linear operators. Suppose that there are a dense subset  $Y_0 $ of  $Y $ and mappings  $S_n: Y_0 \rightarrow X, n \in \mathbb{N} $, such that:
\begin{enumerate}
    \item  $\sum_{n=1}^{\infty} T_{k+n} S_k y $ converges unconditionally in  $Y $, uniformly in  $k \in \mathbb{N} $, for all  $y \in Y_0 $;
    \item  $\sum_{n=1}^{\infty} T_k S_{k+n} y $ converges unconditionally in  $Y $, uniformly in  $k \in \mathbb{N} $, for all  $y \in Y_0 $;
    \item  $\sum_{n=1}^{\infty} S_n y $ converges unconditionally in  $X $, for all  $y \in Y_0 $;
    \item  $T_n S_n y \rightarrow y $ for all  $y \in Y_0 $.
\end{enumerate}
Then the sequence  $\left(T_n\right) $ is frequently universal.
\end{lem}
\begin{thm}\label{theorem 3.2}
Let  $T $ be a (possibly unbounded) operator in a separable Banach space  $X $. Define  $D(T^{\infty}):=\bigcap_{n=0}^{\infty}D(T^{n}) $. Suppose that
\begin{enumerate}
    \item There is a set  $Y_0\subset D( T^\infty ) $ , denses in  $X $, and a mapping  $S: Y_0\to Y_0 $ such that
\begin{enumerate}
    \item  $TSy= y, $ for all  $y\in Y_0; $
    \item $\sum_{n=1}^{\infty} S^n y $ converges unconditionally, for all  $y \in Y_0 $;
    \item $\sum_{n=1}^{\infty} T^n y $ converges unconditionally, for all  $y \in Y_0 $,
\end{enumerate}
    \item there is an operator  $C\in B(X)  $ such that
\begin{enumerate}
    \item  $\operatorname{Im}(C)\subset D(T^\infty) $ and  $T^nC\in B( X) \textit{ for all }n\in \mathbb{N} ; $
    \item  $C $ commutes with  $T $ on  $D( T^{\infty }) ; $
    \item  $\operatorname{Im}(C) $ is dense in  $X. $
\end{enumerate}
\end{enumerate}
Then  $T $ is frequently hypercyclic.
\end{thm}
\begin{proof}
First, we verify that $[\,\text{Im}(C)\,]$ is a Banach space equipped with the norm  
$$
\|x\|_{[\,\text{Im}(C)\,]} := \inf \left\{ \|y\| \mid y \in X, \, Cy = x \right\}.
$$  
For any $x, y \in [\,\text{Im}(C)\,]$, there exist $z, w \in X$ such that $x = Cz$ and $y = Cw$. Therefore, for all $\alpha, \beta \in \mathbb{C}$,  
$$
\alpha x + \beta y = C\left( \alpha z + \beta w \right) \in [\,\text{Im}(C)\,],
$$  
proving that $[\,\text{Im}(C)\,]$ is a linear subspace.  

Positive Definiteness: $\|x\|_{[\,\text{Im}(C)\,]} \geq 0$, and $\|x\|_{[\,\text{Im}(C)\,]} = 0 \Leftrightarrow x = 0$.  

Homogeneity: For any $a \in \mathbb{C}$,  
$$
\|a x\|_{[\,\text{Im}(C)\,]} = \inf \left\{ \|a y\| \mid y \in X, \, Cy = x \right\} = |a| \|x\|_{[\,\text{Im}(C)\,]}.  
$$  

Subadditivity:
$$
\begin{aligned} 
\|x + y\|_{[\,\text{Im}(C)\,]} &\leq \inf \left\{ \|z + w\| \mid z, w \in X, \, Cz = x, \, Cw = y \right\} \\ 
&\leq \inf \left\{ \|z\| + \|w\| \mid z, w \in X, \, Cz = x, \, Cw = y \right\} \\ 
&= \|x\|_{[\,\text{Im}(C)\,]} + \|y\|_{[\,\text{Im}(C)\,]}. 
\end{aligned}
$$  

Thus, $[\,\text{Im}(C)\,]$ is indeed a normed linear space.  

Next, consider a Cauchy sequence $\{x_n\}$ in $[\,\text{Im}(C)\,]$. By definition, for any $\varepsilon > 0$, there exists $N \in \mathbb{N}$ such that for $m, n > N$,  
$$
\|x_m - x_n\|_{[\,\text{Im}(C)\,]} < \frac{\varepsilon}{2}.  
$$  
For each $x_n$, there exists $y_n \in X$ such that $Cy_n = x_n$ and  
$$
\|y_m - y_n\| \leq \|x_m - x_n\|_{[\,\text{Im}(C)\,]} + \frac{\varepsilon}{2} < \varepsilon.  
$$  
Hence, $\{y_n\}$ is a Cauchy sequence in $(X, \|\cdot\|)$. Therefore, there exists $y \in X$ such that $y_n \to y$.  

Since $\|x_n - Cy\|_{[\,\text{Im}(C)\,]} \leq \|y_n - y\|$, it follows that $x_n \to Cy$ in $[\,\text{Im}(C)\,]$, and $Cy \in [\,\text{Im}(C)\,]$. This proves that $\left([\,\text{Im}(C)\,], \|\cdot\|_{[\,\text{Im}(C)\,]}\right)$ is a Banach space.  

Define the operators $ T_n := T^n|_{[\mathrm{Im}(C)]}: [\mathrm{Im}(C)] \to X $, $ n\in \mathbb{N} $. According to (2)(a), these operators are well-defined and continuous. If one can prove that there exists a vector $ x \in \operatorname{Im}(C) $ such that for every non-empty open subset $ U $ of $ X $, the set $ \left\{ n \in \mathbb{N}: T_n x \in U \right\} $ has positive lower density, then $ T $ is frequently hypercyclic.
According to Lemma \ref{lemma 3.1}, it suffices to verify the following condition: there exists a dense subset $ Z_0 \subset X $ and mappings $ S_n: Z_0 \to [\mathrm{Im}(C)] $, $n\in\mathbb{N} $ such that
\begin{enumerate}
    \item $T_nS_nz=z $, for all  $z\in Z_0 $;
    \item $\sum_{n=1}^{\infty} S_n z $ converges unconditionally in  $X $, for all  $z \in Z_0 $; 
    \item $\sum_{n=1}^{\infty} T_{k+n} S_k z $ converges unconditionally in  $X $, uniformly in  $k \in \mathbb{N} $, for all  $z \in Z_0 $;
    \item $\sum_{n=1}^{\infty} T_k S_{k+n} z $ converges unconditionally in  $X $, uniformly in  $k \in \mathbb{N} $, for all  $z \in Z_0 $;
\end{enumerate}
Set  $Z_0:=C(Y_0) $ ; to define  $S_n:Z_0\to[\operatorname{Im}(C)] $ we choose, for each  $z\in Z_0=C(Y_0) $, a fixed vector  $y\in Y_0 $ with  $z=Cy $ and set  $S_nz:=CS^ny $ for  $n\in\mathbb{N} $\cite{martinez2002chaos}.
We show that the conditions (1)-(4) are satisfed. Since  $Y_0 $ are dense in  $X $ and  $C $ has dense range,  $Z_0 $ is dense in  $X $. For  $z\in Z_0 $ we have  $T_nS_nz=T^nCS^ny=CT^nS^ny=Cy=z $ for each  $n\in\mathbb{N} $,  where we have used that  $Y_0\subset D(T^\infty) $ and that  $C $ commutes with  $T $ on  $D(T^\infty). $ This shows (1). (1)(b) shows that  for every  $\varepsilon>0 $ there is some  $N\geq1 $ such that for every finite set  $F\subset\mathbb{N} $ with  $F\cap\{1,2,\ldots,N\}=\emptyset $ we have
 $$\left\|\sum_{n\in F}S^n y\right\|<\varepsilon. $$ And hence  $$\begin{aligned}\left\|\sum_{n\in F}S_n z\right\|_{[\mathrm{Im}(C)]}&=\left\|\sum_{n\in F}CS^ny\right\|_{[\mathrm{Im}(C)]}\\&=\inf\{\|x\|\mid x\in X,\:Cx=\sum_{n\in F}C S^n y\}\\&\leq\left\|\sum_{n\in F}S^n y\right\|<\varepsilon.\end{aligned} $$
This shows (2).And $$\begin{aligned}\left\|\sum_{n\in F}T_kS_{k+n} z\right\|_{[\mathrm{Im}(C)]}&=\left\|\sum_{n\in F}T^k C S^{k+n} y\right\|_{[\mathrm{Im}(C)]}\\&=\left\|\sum_{n\in F}CS^ny\right\|_{[\mathrm{Im}(C)]}\\&=\inf\{\|x\|\mid x\in X,\:Cx=\sum_{n\in F}C S^n y\}\\&\leq\left\|\sum_{n\in F}S^n y\right\|<\varepsilon.\end{aligned} $$
This shows (4).
(1)(c) shows that  for every  $\varepsilon>0 $ there is some  $N\geq1 $ such that for every finite set  $F\subset\mathbb{N} $ with  $F\cap\{1,2,\ldots,N\}=\emptyset $ we have
 $$\left\|\sum_{n\in F}T^n y\right\|<\varepsilon. $$ And hence  $$\begin{aligned}\left\|\sum_{n\in F}T_{k+n}S_k z\right\|_{[\mathrm{Im}(C)]}&=\left\|\sum_{n\in F}T^{k+n}CS^k y\right\|_{[\mathrm{Im}(C)]}\\&=\left\|\sum_{n\in F}CT^ny\right\|_{[\mathrm{Im}(C)]}\\&=\inf\{\|x\|\mid x\in X,\:Cx=\sum_{n\in F}C T^n y\}\\&\leq\left\|\sum_{n\in F}T^n y\right\|<\varepsilon.\end{aligned} $$
This shows (3). This completes the proof.
\end{proof}
While we could consider semigroups of unbounded operators as broadly defined by Hughes\cite{hughes1977semigroups}, it will be more practical to assume that a semigroup  $\{e^{tA}\}_{t\geq0} $ can be regularized, as following:
\begin{defi}[\cite{delaubenfels2003chaos}]
For a bounded injective operator  $C $, the strongly continuous family  $\{W(t)\}_{t\geq0} $ in  $B(X) $ is a  $C $ -regularized semigroup if
l
\begin{enumerate}
    \item $W( 0) = C; $
    \item $W( t) W( s) = CW( t+ s) $ for all  $s,t\geq0. $
\end{enumerate}
An operator  $A $ is the generator of the semigroup  $\{W(t)\}_t\geq0 $ if
 $$Ax\:=\:C^{-1}\left[\lim\limits_{t\to0}\frac{1}{t}\left(W(t)x-Cx\right)\right] $$
with
 $$D(A)\::=\:\{x\mid\text{the limit exists and is in Im}(C)\}\:. $$
\end{defi}

When $A$ generates a $C$-regularized semigroup $\{W(t)\}_{t \geq 0}$, we define the operator semigroup $\left\{e^{t A}\right\}_{t \geq 0}$ by
$$
e^{t A} := C^{-1} W(t), \quad t \geq 0
$$
with domain
$$
D\left(e^{t A}\right) := \left\{ x \in X \mid W(t)x \in \operatorname{Im}(C) \right\}.
$$

We shall utilize the relationship between semigroups and the abstract Cauchy problem:
$$
\left\{\begin{array}{ll}
\frac{\mathrm{d}}{\mathrm{d}t} u(t,x) = A(u(t,x)), & t \geq 0 \\
u(0,x) = x 
\end{array}\right.
$$

As stated in \cite{2}, we define the solution space $Z=Z(A)$ for $A$ as the collection of all $x \in X$ satisfying: the corresponding equation has a mild solution, i.e., the mapping $t \mapsto u(t,x)$ belongs to $C([0,\infty),X)$, such that for every $t \geq 0$, $v(t,x)=\int_0^t u(s,x)\,\mathrm{d}s \in D(A)$, and the following holds:
$$
\frac{\mathrm{d}}{\mathrm{d}t} v(t,x) = A(v(t,x)) + x, \quad t \geq 0
$$

This implies
$$
Z(A) = \left\{ x \in X \,\bigg|\,  \ \forall t \geq 0, W(t)x \in \operatorname{Im}(C),\ \text{and the map } t \mapsto C^{-1}W(t)x \text{ is continuous}
 \right\}
$$

For $x \in Z(A)$, its mild solution is given by:
$$
u(t,x) = e^{tA}x, \quad t \geq 0
$$

Note that $Z(A) \subset \bigcap_{t \geq 0} D(e^{tA})$. The Fr\'{e}chet space topology on $Z(A)$ can be constructed through the family of seminorms
$$
\|x\|_n := \sup_{0 \leq s \leq n} \|u(s,x)\|, \quad n \in \mathbb{N},\ x \in Z(A).
$$
Now, we can define frequent hypercyclicity for semigroups of unbounded operators.

We recall that the lower density of a measurable set  $M \subset \mathbb{R}_{+} $is defined by
 $$
\underline{\operatorname{Dens}}(M):=\liminf _{N \rightarrow \infty} \mu(M \cap[0, N]) / N
 $$
where  $\mu $ is the Lebesgue measure on  $\mathbb{R}_{+} $. 

Then the Frequent Hypercyclicity can be extended to  $C $-regularized semigroup, as following:
\begin{defi}
Let A be the generator of a C-regularized semigroup. Then the semigroup $\{e^{tA}\}_{t\geq0} $ is called frequently hypercyclic if there exists a vector  $x\in Z(A) $ for every non-empty open subset  $U $ of  $X $,  $$\underline{\operatorname{Dens}}\left\{t\geq0: e^{tA}x \in U\right\}>0.$$
\end{defi}
\begin{thm}\label{theorem 3.5}
Let A be the generator of a C-regularized semigroup. If there exists  $t>0 $ such that the operator  $e^{tA} $ is frequently hypercyclic, then  $\{e^{tA}\}_{t\geq0} $ is frequently hypercyclic.
\end{thm}
\begin{proof}
Assume w.l.o.g. that  $t=1 $ and let  $x $ be a frequently hypercyclic vector for  $e^A $. Let  $y \in X, U, V $ 0-neighbourhoods such that  $V+V \subseteq U $. By the strong continuity of  $\{e^{tA}\}_{t\geq0} $, there exists  $0<\delta<1 $ such that  $e^{sA} y-y \in V $ for every  $s \in[0, \delta] $.
Moreover, by the local equicontinuity of  $\{e^{tA}\}_{t\geq0} $, there exists a 0-neighbourhood  $V^{\prime} $ such that  $e^{sA}\left(V^{\prime}\right) \subseteq V $ for every  $s \in[0, \delta] $.
By assumption,
 $$
\underline{\text { dens }}\left\{n \in \mathbb{N}: e^{nA} x \in y+V^{\prime}\right\}>0
 $$
If  $e^{nA} x \in y+V^{\prime} $, then for every  $t \in[n, n+\delta] $
 $$
e^{tA} x-y=e^{(t-n)A}\left(e^{nA} x-y\right)+e^{(t-n)A} y-y \in e^{(t-n)A}\left(V^{\prime}\right)+V \subseteq V+V \subseteq U
 $$
Thus, for every  $N \in \mathbb{N} $
 $$
\frac{\mu\left\{t \leq N: e^{tA} x \in y+U\right\}}{N} \geq \delta \frac{\#\left\{n \in \mathbb{N}: n \leq N, e^{nA} x \in y+V^{\prime}\right\}}{N}
 $$
hence
 $$
\liminf _{N \rightarrow \infty} \frac{\mu\left\{t \leq N: e^{tA} x \in y+U\right\}}{N} \geq \delta \liminf _{N \rightarrow \infty} \frac{\#\left\{n \in \mathbb{N}: n \leq N, e^{nA} x \in y+V^{\prime}\right\}}{N}>0 .
 $$
\end{proof}

\begin{thm}
Let  $A $ be the generator of a C-regularized semigroup. Suppose that there is a set  $Y_0\subset
Z( A)  $, dense in  $X $, and a mapping  $S: Y_0\to Y_0 $ such that
\begin{enumerate}
    \item  $e^ASy= y, \textit{for all }y\in Y_0; $
    \item  $\sum_{n=1}^{\infty} S^ny $ converges unconditionally, for all  $y\in Y_0; $
    \item  $\sum_{n=1}^{\infty} e^{nA }y $ converges unconditionally, for all  $y\in Y_0 $,
\end{enumerate}
If, in addition, Im $(C) $ is dense in  $X $ then the operator e $^A $ and the semigroup
 $ \{e^{tA}\} _{t\geq 0} $ are frequently hypercyclic. 
\end{thm}
\begin{proof}
Restricting the operator $e^{A}$ to the subspace $Z = Z(A)$, denoted by $T := e^{A}|_{Z}$, we obtain:
\begin{enumerate}
    \item Since $e^{A}$ maps $Z$ into itself, we have $D(T^\infty) = Z$
    \item For any $C$-regularized semigroup, $C$ commutes with $e^{A}$ on $Z$, and $\operatorname{Im}(C) \subset Z$
    \item For each $n \in \mathbb{N}$, the operator $e^{nA}C = W(n)$ is continuous
\end{enumerate}

From the exponential relationship $(e^{A})^{n} = e^{nA}$ valid on $Z$, Theorem \ref{theorem 3.2} implies that $e^{A}$ is frequently hypercyclic. Further applying Theorem \ref{theorem 3.5}, we obtain the frequent hypercyclicity of the semigroup $\{e^{tA}\}_{t \geq 0}$.

\end{proof}
Consider the function space $X = C_0(\mathbb{R}^+)$, consisting of all continuous functions on the positive real axis vanishing at infinity. Define the operator $C$ as:
$$
C f(x) = f(x)
$$
i.e., the identity operator on $X$, which is clearly a bounded injective operator.

Next, define the semigroup $W(t)$ by:
$$
W(t) f(x) = e^{\lambda t} f(x+t), \quad \lambda > 0
$$
We verify whether $W(t)$ satisfies the conditions of a $C$-regularized semigroup:

\begin{enumerate}
    \item $W(0) = C$:\\
    $W(0) f(x) = C f(x)$, which satisfies the condition.
    
    \item $W(t) W(s) = C W(t+s)$:\\
    $W(t) W(s) f(x) = W(t)[e^{\lambda s} f(x+s)] = e^{\lambda t}e^{\lambda s} f(x+s+t)$.\\
    On the other hand, $C W(t+s) f(x) = C[e^{\lambda (t+s)}f(x+t+s)] = e^{\lambda t}e^{\lambda s} f(x+s+t)$.\\
    Hence, $W(t) W(s) = C W(t+s)$, satisfying the condition.
\end{enumerate}

The generator $A$ is defined by:
$$
A f = C^{-1}\left[\lim_{t \rightarrow 0} \frac{W(t) f - C f}{t}\right]
$$
Calculating this limit:
$$
\lim_{t \rightarrow 0} \frac{W(t) f - C f}{t} = \lim_{t \rightarrow 0} \frac{e^{\lambda t} f(x+t) - C f(x)}{t}
$$
Since $C f(x) = f(x)$, we have:
$$
\lim_{t \rightarrow 0} \frac{e^{\lambda t} f(x+t) - f(x)}{t}
$$
Expanding and simplifying:
$$
= \lim_{t \rightarrow 0} \frac{f(x+t) - f(x) + \lambda t f(x+t)}{t} = f^{\prime}(x) + \lambda f(x)
$$
Therefore,
$$
A f(x) = C^{-1}\left(f^{\prime}(x) + \lambda f(x)\right) = f^{\prime}(x) + \lambda f(x)
$$
with domain:
$$
D(A) = \left\{f \in C_0^1(\mathbb{R}^+) \mid f^{\prime}(x) + \lambda f(x) \in \operatorname{Im}(C)\right\}
$$

The solution space $Z(A)$ is:
$$
Z(A) = \left\{ f \in X \,\bigg|\,  \ \forall t \geq 0, W(t)f \in \operatorname{Im}(C),\ \text{and the map } t \mapsto C^{-1}W(t)f \text{ is continuous} \right\}
$$

Now we verify the conditions for frequent hypercyclicity criteria:
\begin{enumerate}
    \item Dense subset $Y_0 \subset Z(A)$:\\
    Take $Y_0 = C_c^{\infty}(\mathbb{R}^+)$, the space of smooth functions with compact support, which is dense in $C_0(\mathbb{R}^+)$.
    
    \item Map $S: Y_0 \rightarrow Y_0$ satisfying $e^A S y = y$:\\
    Define $S$ by: $S y(x) = e^{-\lambda} y(x-1)$.\\
    Verify $e^A S y = y$:\\
    $e^A y(x) = e^{\lambda} y(x+1)$.\\
    $S y(x) = e^{-\lambda} y(x-1)$.\\
    Thus, $e^A S y(x) = y(x)$, satisfying the condition.
    
    \item Unconditional convergence of series $\sum_{n=1}^{\infty} S^n y$ and $\sum_{n=1}^{\infty} e^{n A} y$:\\
    For $y \in Y_0$ with compact support, $\sum_{n=1}^{\infty} e^{n A} y$ is finite and thus unconditionally convergent.\\
    $\sum_{n=1}^{\infty} S^n y = \sum_{n=1}^{\infty} e^{-n\lambda} y(x-n)$ converges unconditionally due to exponential decay $e^{-n\lambda}$ and finite support of $y(x-n)$.
    
    \item Density of $\operatorname{Im}(C)$ in $X$:\\
    $\operatorname{Im}(C) = X$ is trivially dense.
\end{enumerate}

In conclusion, this example satisfies all conditions and provides a concrete frequently hypercyclic semigroup conforming to the theorems.

\end{document}